\documentclass [smallextended,numbook,final]{svjour3}
\usepackage{thmtools,thm-restate}

 \usepackage[numbers,square,comma]{natbib}
  \usepackage{bm}
\usepackage{amsmath}
\usepackage{amssymb}
\usepackage{enumitem}
\newcommand{\Real}{\mathbb{R}}

\usepackage{url}
\usepackage{color}
\newcommand{\set}[1]{\left\{#1\right\}}                     
\newcommand{\norm}[1]{\left\|#1\right\|}
\newcommand{\bra}[1]{\left(#1\right)}
\usepackage{stmaryrd}
\newcommand{\sidx}[1]{\left\llbracket     #1 \right\rrbracket}
\usepackage{comment}
\DeclareMathOperator{\cl}{cl}
\DeclareMathOperator{\conv}{conv}

\DeclareMathOperator{\relint}{relint}

\DeclareMathOperator{\Int}{int}
\DeclareMathOperator{\dom}{dom}
\DeclareMathOperator{\proj}{proj}
\definecolor{darkred}{rgb}{0.7,0,0}
\definecolor{lightred}{rgb}{0.5,0,0}
\definecolor{lightgreen}{rgb}{0,0.8,0}

\newcommand{\bx}[0]{\bm{x}}
\newcommand{\by}[0]{\bm{y}}
\newcommand{\bz}[0]{\bm{z}}

\usepackage[colorlinks=true,breaklinks=true,bookmarks=true,urlcolor=blue,
     citecolor=blue,linkcolor=blue,bookmarksopen=false,draft=false]{hyperref}

\begin{document}
\title{Shapes and recession cones in mixed-integer convex representability}
\author{Ilias Zadik \and Miles Lubin \and Juan Pablo Vielma}
\date{\today}

\institute{
Ilias Zadik \at Center for Data Science \\ New York University\\
\email{zadik@nyu.edu}
\and
Miles Lubin \at Google Research \\
\email{mlubin@google.com}
\and
Juan Pablo Vielma \at Google Research  \\
\email{jvielma@google.com}
}

\maketitle

\begin{abstract}

Mixed-integer convex representable (MICP-R) sets are those sets that can be represented exactly through a mixed-integer convex programming formulation. Following up on recent work by Lubin et al. (2017, 2020) we investigate structural geometric properties of MICP-R sets, which strongly differentiate them from the class of mixed-integer linear representable sets (MILP-R).  First, we provide an example of an MICP-R set which is the countably infinite union of convex sets with countably infinitely many different recession cones. This is in sharp contrast with MILP-R sets which are at most infinite unions of polyhedra that share the same recession cone. Second, we provide an example of an MICP-R set which is the countably infinite union of polytopes all of which have different shapes (no pair is combinatorially equivalent, which implies they are not affine transformations of each other). Again, this is in sharp contrast with MILP-R sets which are at most infinite unions of polyhedra that are all translations of a finite subset of themselves. Interestingly, we show that a countably infinite union of convex sets sharing the same volume can be MICP-R only if the sets are all translations of a finite subset of themselves (i.e. the natural conceptual analogue to the MILP-R case).
\end{abstract}

\keywords{Mixed-Integer Convex Optimization, Mixed-Integer Programming Formulations}
\subclass{90C11, 90C25}

\section{Introduction}

We use the following definition of mixed-integer convex representability and mixed-integer convex programming (MICP) formulations.
\begin{definition}\label{MICPFORMULATIONDEF}
   Let $n,p,d \in \mathbb{N}$, $S \subseteq \mathbb{R}^n$ and $M\subseteq \mathbb{R}^{n+p+d}$ be a closed convex set. We say that  $M$ induces an \textbf{MICP formulation} of $S$ if
   \begin{equation}\label{MICPdefexists}
   \bx \in S \quad \Leftrightarrow \quad \exists \by \in \mathbb{R}^{p}, \bz \in \mathbb{Z}^{d}  \quad \text{s.t.} \quad \bra{\bx,\by,\bz}\in  M.
   \end{equation}

   We call a set $S \subseteq \mathbb{R}^n$ \textbf{MICP representable} (MICP-R) if there exists a closed convex set $M\subseteq \mathbb{R}^{n+p+d}$ that induces an MICP formulation of $S$. 
   
   We further call $S$ \textbf{binary MICP-R} if there exists such an $M$ that additionally satisfies  \[\proj_{\bz}\bra{M\cap \bra{\mathbb{R}^{n+p} \times \mathbb{Z}^d}}\subseteq\set{0,1}^d,\]where $\proj_{\bz}$ is the projection onto the last $d$ variables (i.e. informally if $S$ has an MICP formulation with only binary integer variables). 
    
   If $S$ is MICP-R, but not binary MICP-R we call it \textbf{general-integer MICP representable} (general-integer MICP-R) to emphasize the need for unbounded integer variables to model it (note that any MICP formulation with only bounded integer variables can be converted to one with only binary variables through a standard affine transformation  \cite[Footnote 4]{MORpaper})
\end{definition} 

It is an easy corollary of \autoref{MICPFORMULATIONDEF} (see e.g. \cite[Theorem 4.1]{MORpaper}) that if $M\subseteq \mathbb{R}^{n+p+d}$ induces an MICP formulation of $S\subseteq \mathbb{R}^{n}$, then
   \begin{equation}\label{infiniteunion}
   S=\bigcup_{\bm{z} \in I\cap \mathbb{Z}^d} \proj_{\bx}\bra{B_{\bm{z}}},\end{equation}
   where $I=\proj_{\bz}\bra{M}$ and  $B_{\bm{z}}=M \cap (\mathbb{R}^{n+p} \times \{\bm{z}\})$ for any $\bm{z} \in I$. Hence, a binary MICP-R set is a finite union of  convex sets, and a general-integer MICP-R set is a countably infinite union  of convex sets (in both cases the convex sets are projections of closed convex sets, but in principle may not be closed themselves). The work in \cite{Lubin2016b,MORpaper} focuses on understanding the specific structure imposed on these unions of convex sets by an MICP formulation. 
   
With regards to the finite unions for the binary MICP-R case,  MICP formulations proposed prior to \cite{Lubin2016b,MORpaper} required the convex sets in the union \eqref{infiniteunion} to have the same recession cones (e.g. \cite{Ceria1999,DBLP:conf/ipco/PiaP16,Jeroslow1984,jeroslow1987representability,stubbs1999branch}). However, \cite{Lubin2016b,MORpaper} show that such a restriction on the recession cones is not necessary for a finite union of convex sets to be binary MICP-R (e.g.
\cite[Proposition 4.2]{MORpaper}). 
For instance, as illustrated in \cite[Example 4.1]{MORpaper}, the set $S=\set{0}\cup (-\infty,-1]$ is the union of two convex sets  ($\set{0}$ and $(-\infty,-1]$) that have different recession cones ($\set{0}$ and $(-\infty,0]$ respectively) and a binary MICP formulation of $S$ is induced by 
\[M= \set{\bra{x,y,z}\in \Real \times \Real^2_+\,:\, x^2 \le yz,\quad x \leq -z,\quad 0 \leq z \leq 1}.\]
The first question we consider in this paper is whether a conceptual analogue of such a recession cone property extends to the class of general-integer MICP-R sets; i.e. \textit{are there general-integer MICP-R sets that are countable infinite unions of convex sets with countably infinitely many different recession cones?} We answer this question in the affirmative. 

On the other hand, the fact that the union of these convex sets defines a general-integer MICP-R set implies a certain structure on them, even when all the convex sets are bounded (and therefore have trivial recession cones). For instance, \cite{MORpaper} shows that the countably infinite union of line segments with increasing slopes (all of which are affinely equivalent)  given by the set $\bigcup_{i\in \mathbb{Z}}\conv\bra{\set{\bra{i,i^2},\bra{i+1,(i+1)^2}}}$ is not MICP-R. The second question we consider in this work is on the ``shapes" of the convex sets defined in the union \eqref{infiniteunion}, where equality of shapes is defined by some precise notion such as affine equivalency. Here we ask, \textit{are there general-integer MICP-R sets that are countably infinite unions of convex sets with countably infinitely many different shapes?} We also answer this question in the affirmative for a shape-equivalency  notion that is weaker than affine equivalency (it is implied by affine equivalency). In addition, we show that if certain volume conditions hold, then the union must have finite shapes for a shape-equivalency notion that is stronger than affine equivalency (it implies affine equivalency).

For both questions we will also consider a restriction of general-integer MICP-R sets that was introduced in \cite{MORpaper} to avoid certain pathological behavior of such sets (see \cite[Section 1.3]{MORpaper} for an example). This restriction is defined as follows.

   \begin{definition}\label{rationaldefc}
    We say a set $I \subseteq \mathbb{R}^d$ is \textbf{rationally unbounded} if for any rational affine image  $I'\subseteq \Real^{d'}$ of $I$, either $I'$ is a bounded set or it holds that $ \mathbb{Z}^{d'} \cap (I'_\infty \setminus \set{\bm{0}}) \not = \emptyset$, where $C_\infty =\set{\bm{r}\in \Real^d\,:\, \bm{x}+\lambda \bm{r} \in C \quad \forall \bm{x} \in C,\quad \lambda\geq 0}$ for any convex set $C$.
  
  We say that a set $S$ is \textbf{rational MICP representable} (rational MICP-R) if it has an MICP representation induced by the set $M$ such that $\proj_{\bz}(M)$ is rationally unbounded.  
\end{definition}
  
Under mild assumptions, rational MICP-R sets can always be written as the union of a finite family of compact convex sets and a finite family of closed periodic sets \cite[Theorem 5.1]{MORpaper}. However, a single periodic set, in principle, could be itself a countable union of convex sets with different recession cones and/or shapes (e.g. for the definition of periodic set in \cite[Definition  1.7]{MORpaper} we have that $S\times \mathbb{Z}$ is periodic for any set $S$). Of course, it is not obvious to what extent such periodic set could arise within a rational MICP-R set. Hence, while rational MICP-R sets add an additional non-trivial structure to the studied sets, the posed questions are not automatically answered even in this restricted case.

Similarly to \cite{Lubin2016b,MORpaper} it will be convenient to define the following sets associated to an MICP formulation. 
  \begin{definition}\label{def:index}
      Let $M\subseteq \mathbb{R}^{n+p+d}$ be a closed, convex set that induces an MICP formulation of $S \subseteq \mathbb{R}^n$. We refer to $I=\proj_{\bz}\bra{M}$ as the \textbf{index set} of the MICP formulation and to {the collection of sets $\set{A_{\bm{z}}}_{\bz\in I}$ with $A_{\bm{z}}=\proj_{\bx}\bra{M \cap (\mathbb{R}^{n+p} \times \{\bm{z}\})}$ for each $\bz \in I$, as its \textbf{$\bm{z}$-projected sets}. }
  \end{definition}Note that under this definition the convex sets defining the union \eqref{infiniteunion} correspond exactly to the family of the $\bm{z}$-projected sets.

\section{Infinitely many recession cones}
Towards providing some intuition for the desired example, we start with describing certain potentially fundamental differences between the binary and general-integer MICP-R classes. Let $M\subseteq \mathbb{R}^{n+p+d}$ be a closed convex set that induces a binary MICP formulation of $S\subseteq \Real^n$ and let $I\subseteq \Real^d$ be its index set. Because $M$ induces a binary MICP formulation, notice that $I\cap\mathbb{Z}^d \subseteq \set{0,1}^d$ and hence $M\cap \bra{\Real^{n+p}\times \conv\bra{I\cap \mathbb{Z}^d}}$ also induces a binary MICP formulation of $S$.  
Then, an arguably special property of binary MICP formulations is that we may assume that their index sets $I$ are such that all elements of $I\cap\mathbb{Z}^d$ lie on the boundary of $I$. 

Notice that this property no longer holds for general-integer MICP formulations;  e.g., consider the case $p=0$, $n=d$, $M=\set{\bra{\bx,\bz}\in \Real_+^{2d}\,:\, \bx=\bz}$, $I=\Real_+^{d}$ and $S=\mathbb{Z}_+^d$ (where $\mathbb{Z}_+$ is the set of non-negative integers). Hence, we may easily construct an example where we have an infinite number of elements of $I\cap\mathbb{Z}^d$ that are in the relative interior of $I$. One can naturally conjecture that this could be a crucial property allowing for the corresponding infinitely many $\bm z$-projected sets to define countably infinitely many different recession cones. On the contrary, the following proposition shows that such \emph{interior} integer elements of $I$ define necessarily the same recession cone (up to taking closures) and therefore cannot be of help to construct a general-integer MICP-R set with an infinite number of recession cones.

\begin{restatable}{proposition}{interiorrecessioncones}\label{interiorrecessionconesref}
 Let $M\subseteq \Real^{n+p+d}$ be a closed convex set inducing an MICP formulation of a set  $S\subseteq \Real^n$, $I$ be its index set, and $\set{A_{\bz}}_{\bz\in I}$ be its $\bz$-projected sets. Then
 \[\bra{\cl\bra{A_{\bz}}}_{{\infty}}=\bra{\cl\bra{A_{\bz'}}}_{{\infty}}\quad \forall \bz,\bz'\in \relint\bra{I}.\]
\end{restatable}
\begin{proof}
 For any $A\subseteq \Real^n$ and let $\sigma_A:\Real^n\to\Real\cup \set{+\infty}$ be the support function of $A$ given by $\sigma_A\bra{\bm{c}}=\operatorname{sup}\{ {\bm{c}}^T\bm{x} : \bm{x} \in A \}$, and for any $\bm{c}\in \Real^n$ let $f_{{\bm{c}}}:I\to \Real\cup \set{-\infty}$ be given by $f_{{\bm{c}}}(\bm{z})=-\sigma_{A_{\bm{z}}}\bra{-\bm{c}} =  \operatorname{inf}\{ {\bm{c}}^T\bm{x} : \bm{x} \in A_{\bm{z}} \}$.  By  \cite[Lemma 6.11]{MORpaper}, $f_{{\bm{c}}}$ is convex for any $\bm{c}\in \Real^n$ even though it may fail to be finite valued.\footnote{For example, consider $M=\set{(x,y,z)\in \Real^3 \,:\,x^2\leq y\cdot z,\quad x\leq -z\,\quad 0\leq z\leq 1}$, which has $I = \proj_z(M)=[0,1]$ and $f_{1}(0)=0$ and $f_{1}(z)=-\infty$ for all $z\in (0,1]$.}
 
 We claim that, for any fixed $\bm{c}\in \Real^n$, if there exists $\bar{\bz}\in \relint\bra{I}$ such that $f_{{\bm{c}}}(\bar{\bz})> -\infty$, then  $f_{{\bm{c}}}({\bz})> -\infty$ for all $\bz \in I$. Indeed, if such $\bar{\bz}$ exists, for any $\bz\in I$ there exists $\varepsilon>0$ such that $\bar{\bz}^+:= \bar{\bz} + \varepsilon (\bz-\bar{\bz})\in I$ and $\bar{\bz}^-:= \bar{\bz} - \varepsilon (\bz-\bar{\bz})\in I$. By convexity of $f_{{\bm{c}}}$ we have $-\infty <f_{{\bm{c}}}(\bar{\bz})\leq\frac{1}{2}\bra{f\bra{\bar{\bz}^+}+f_{{\bm{c}}}\bra{\bar{\bz}^-}}$ and hence $f_{{\bm{c}}}\bra{\bar{\bz}^+}, f_{{\bm{c}}}\bra{\bar{\bz}^-}>-\infty$. Again by convexity we have $-\infty <f_{{\bm{c}}}(\bar{\bz}^+)\leq (1-\varepsilon) f_{{\bm{c}}}\bra{\bar{\bz}}+\varepsilon f_{{\bm{c}}}\bra{\bz}$ so $f_{{\bm{c}}}\bra{\bz}>-\infty$.

The claim in the previous paragraph shows that $\dom\bra{\sigma_{A_{\bm{z}}}}=\dom\bra{\sigma_{A_{\bm{z}'}}}$ for all $\bz,\bz'\in \relint\bra{I}$, where  $\dom\bra{\sigma_{A}}=\set{\bm{c}\in \Real^n\,:\, \sigma_{A}(\bm{c})<\infty}$ for any  $A\subseteq \Real^n$. Combining this with   \cite[Proposition C.2.2.1]{hiriart-lemarechal-2001}, we then have that $\dom\bra{\sigma_{\cl\bra{A_{\bm{z}}}}}=\dom\bra{\sigma_{\cl\bra{A_{\bm{z}'}}}}$ for all $\bz,\bz'\in \relint\bra{I}$. The final result follows by applying \cite[Proposition C.2.2.4]{hiriart-lemarechal-2001}, which states that for any closed convex set $C$ we have that $\cl\bra{\dom\bra{\sigma_{C}}}$ and $C_\infty$ are mutually polar cones. \qed
\end{proof}

\autoref{interiorrecessionconesref} establishes that if it is possible to construct a general-integer MICP-R set that is a countably infinite union of convex sets with infinitely many different recession cones, we need to focus on $A_{\bm z}$ where $\bm z$ lies on the boundary of the index set. In the following lemma, we show that we can indeed obtain a countably infinite number of different recession cones. Interestingly, our construction applies even to the restricted class of general \textit{rational} MICP-R sets. 

\begin{lemma}[Infinite Recession Cones]\label{infiniterecesionexample}
Let $M\subseteq \Real^{7}$ be the closed convex set on variables $\bx\in \Real^4$, $y\in \Real$ and $\bz\in \Real^2$ defined by the following inequalities:
\begin{subequations}
\begin{alignat}{3}
\bra{x_4-\bra{1-\frac{1}{1+c}}x_3}^2&\leq \bra{z_2 -c^2-2c\bra{z_1-c}}y &\quad&\forall c\in \mathbb{Z}_+\label{infinite_c_eq}\\
0\leq  x_4&\leq x_3,\label{noninfinite_c_eq:start}\\
x_i&=z_i&\quad&\forall i\in \sidx{2} \label{infinite_c_eq:end} \\
0\leq z_1,\quad z_1^2&\leq z_2.\label{infiniterecesionexample:Ieq}
\end{alignat}
\end{subequations}
 Then $M$ induces a rational MICP formulation of $S=\bigcup_{\bz\in I} A_{\bz}$ for $I=\proj_{\bz}\bra{M}=\set{\bz\in \Real^2_+\,:\, z_1^2\leq z_2}$, 
 \begin{equation}\label{infiniterecesionexample:seq1}
     A_{\bz}=\set{\bx\in \mathbb{R}^4\,:\, x_1=z_1,\quad x_2=z_2,\quad  0\leq  x_4\leq x_3}\quad \forall \bz\in \Int\bra{I}
 \end{equation}
and
\begin{equation}\label{infiniterecesionexample:seq2}
    A_{\bz}=\set{\bx\in \mathbb{R}^4\,:\, \begin{alignedat}{3}x_1&=z_1,\\ x_2&=z_2=z^2_1,\\ 0&\leq x_4=\bra{1-\frac{1}{1+z_1}}x_3\end{alignedat}} \quad \forall \bz\in I\setminus \Int\bra{I}.
\end{equation}

In addition, all $\bz$-projected sets for $\bz\in \Int(I)$ have the same recession cone $\bra{A_{\bz}}_\infty=\set{\bx\in \mathbb{R}^4\,:\, x_1=x_2=0,\quad 0\leq  x_4\leq x_3}$. In contrast, for $\bz\in I\setminus \Int\bra{I}$
we have  \[\bra{A_{\bz}}_\infty=\set{\bx\in \mathbb{R}^4\,:\, x_1=x_2=0,\quad  x_4\geq 0,\quad x_4=\bra{1-\frac{1}{1+z_1}}x_3}.\] Then, $S$ is a countably infinite union of disjoint convex sets with countably infinitely many different recession cones.  
\end{lemma}
\begin{proof}
We begin by showing that $\proj_{\bx,\bz}\bra{M}=Q$ for  
\begin{equation}\label{infiniterecesionexample:firstclaim}
    Q=\set{(\bx,\bz)\in \Real^7\,:\, \begin{alignedat}{3} {z_1^2<z_2}\quad&\vee \quad {x_4=\bra{1-\frac{1}{1+z_1}}x_3},\\
    \text{\eqref{noninfinite_c_eq:start}}&\text{--\eqref{infiniterecesionexample:Ieq}}\end{alignedat}}.
\end{equation}
For this, let $Z=\set{\bz\in \Real^2\,:\, \eqref{infiniterecesionexample:Ieq} }$ and    $l:Z\times \Real_+\to \Real$ given by \[l(\bz,c):=z_2 -c^2-2c\bra{z_1-c}=z_2-z_1^2+(z_1-c)^2.\]
Then, because $z_2-z_1^2\geq 0$ for all $\bz\in Z$, we have that
\begin{equation}\label{infiniterecesionexample:eq1}
    l(c,c^2,c)=0\quad \text{ and }\quad l(\bz,c)\geq 1\quad \forall c\in \mathbb{Z}_+,\bz\in (Z\cap \mathbb{Z}^2)\setminus \set{(c,c^2)}.
\end{equation}
Next, let $X=\set{(x_3,x_4)\in \Real^2\,:\, \eqref{noninfinite_c_eq:start} }$ and $h:X\times \Real_+\to \Real$ given by 
\[h(x_3,x_4,c):=\bra{x_4-\bra{1-\frac{1}{1+c}}x_3}^2.\]
Then for any $(x_3,x_4)\in X$ and  $c\in [0,x_4/(x_3-x_4)]$ we have that $h$ is decreasing in $c$ and hence $h(x_3,x_4,c)\leq h(x_3,x_4,0)=x_4^2$. Similarly, for any $(x_3,x_4)\in X$ and $c\in [x_4/(x_3-x_4),\infty)$  we have that $h$ is increasing in $c$ and hence  $h(x_3,x_4,c)\leq \lim_{c\rightarrow\infty} h(x_3,x_4,c)=(x_4-x_3)^2$. Then,
\begin{equation}\label{infiniterecesionexample:eq2}
    h(x_3,x_4,c) \leq \max\{(x_4-x_3)^2,x_4^2\} \quad \forall  (x_3,x_4)\in X, c\in \mathbb{Z}_+.
\end{equation}
Under this notation we have that $M$ is defined by 
\begin{subequations}\label{infiniterecesionexample:redefM}
   \begin{alignat}{3} h(x_3,x_4,c)&\leq l(\bz,c)y \quad\forall c\in \mathbb{Z}_+,\label{infiniterecesionexample:redefMinf}\\
    (x_3,x_4)&\in X,\quad z\in Z\\ 
    x_i&=z_i\quad\forall i\in \sidx{2},\end{alignat}
\end{subequations}
while $Q$ is defined by 
\begin{subequations}\label{infiniterecesionexample:redefprojM}
   \begin{alignat}{3} {z_1^2<z_2}\quad&\vee \quad {x_4=\bra{1-\frac{1}{1+z_1}}x_3},\label{infiniterecesionexample:redefprojMimplication}\\
    (x_3,x_4)&\in X,\quad z\in Z\\ 
    x_i&=z_i\quad\forall i\in \sidx{2}.\end{alignat}
\end{subequations}
To show $\proj_{\bx,\bz}\bra{M}\subseteq Q$, first note that  \eqref{infiniterecesionexample:eq1} implies  $l(\bz,z_1)=0$ when $z_1^2=z_2$. Hence, if $z_1^2=z_2$, then constraint \eqref{infiniterecesionexample:redefMinf} for $c=z_1$ implies $h(x_3,x_4,z_1)=0$. That is, constraint \eqref{infiniterecesionexample:redefMinf} for $c=z_1$ enforces \eqref{infiniterecesionexample:redefprojMimplication} and hence the containment follows. To show $Q\subseteq \proj_{\bx,\bz}\bra{M}$, note that because of \eqref{infiniterecesionexample:eq2}, for any $(\bx,\bz)\in Q$ we have $(\bx,y,\bz)\in M$ for $y=\max\{(x_4-x_3)^2,x_4^2\}$.

The characterization of the index set $I$ and $\bz$-projected sets $\set{A_{\bz}}_{\bz\in I}$ follows directly from $\proj_{\bx,\bz}\bra{M}=Q$ for $Q$ defined in \eqref{infiniterecesionexample:firstclaim}.
\qed
\end{proof}

\section{Shapes}

One might wonder whether the sets defined in \autoref{infiniterecesionexample} are already defining infinitely many different shapes. Yet, it is a rather immediate observation that the sets $A_{\bz}$ for $\bz\in \Int\bra{I}$ are all translations of each other and sets $A_{\bz}$ for $\bz\in I\setminus \Int\bra{I}$ are all translations and rotations of each other. Hence, $S$ is the union of sets with exactly two shapes under any shape-equivalency notion where shapes are preserved under translations and rotations. This is certainly not the only reasonable shape-equivalency notion, as we could also consider the cases where shapes are only preserved under translations or are further preserved by any invertible affine transformation. The following three shape-equivalency notions will be useful to present and contrast the  results in this section.  

\begin{definition}\label{shapedef}
We consider the following notion of shape equivalency.
\begin{enumerate}
    \item {\bf Translation equivalency:} Two sets have the same shape if they are translations of each other. 
    \item {\bf Affine equivalency:} Two sets have the same shape if they are invertible affine transformations of each other.
    \item {\bf Combinatorial equivalency (for polytopes):} Two polytopes have the same shape if there is a bijection between their \emph{faces} that preserves the inclusion relation \cite{ziegler2012lectures}.
\end{enumerate}
\end{definition}

The shape-equivalency notions in \autoref{shapedef} are listed from strongest (harder for two sets to have equal shape) to weakest (easier for two sets to have equal shape). A relevant class of combinatorially equivalent polytopes which are not all affinely equivalent is the family of all $k$-sided polygons, for any fixed $k\geq 4$ (e.g. \cite[page 6]{ziegler2012lectures}).

In the rest of this section we give an explicit rational MICP-R formulation for a mutually disjoint union of regular polygons with increasing number of sides, which shows that MICP-R sets may be countably infinite unions of polytopes with different shapes under the shape-equivalency notion of combinatorial equivalency. However, we also show that an infinite union of convex sets sharing the same volume can be MICP-R only if they have a finite number of shapes under the shape-equivalency notion of translation equivalency.

\subsection{Rational MICP-R sets can have infinitely many shapes}

To construct the desired formulation, we will construct an infinite union of appropriate polygons which is a rational MICP-R set. We will use the following technical proposition.

\begin{restatable}{proposition}{exampleslemmashapeplus}\label{exampleslemmashaperefplus} 
Let $r:\Real_+\to \Real_+$  and for each $i\in \mathbb{N}\setminus\set{0}$ let $C^i\subseteq \Real^n$ be a closed convex set with $0\in C^i$. If
\begin{enumerate}
    \item $r$ is concave, strictly increasing and $r(z)< 1/2$ for all $z\in \Real_+$, and
    \item\label{exampleslemmashapepluscontainmentconditions} $\set{\bx\in \Real^n\,:\, \norm{\bx}_2\leq \frac{r(i-1)+r(i+1)}{2}}\subseteq C^i \subseteq \set{\bx\in \Real^n\,:\, \norm{\bx}_2\leq r(i)}$ for all $i\in \mathbb{N}\setminus\set{0}$,
\end{enumerate}
then   $S=\bigcup_{i=1}^\infty \bra{C^i+i\bm{e}(1)}$ is a {\bf rational} MICP representable infinite union of mutually disjoint sets.
\end{restatable}
\begin{proof}
 First, notice that for any $i\geq 1$, $\hat C^i:=\operatorname{cl}(\{ (\bm{x},t) : \bm{x}/t \in C^i, t > 0 \})$ is a closed convex cone with the properties,
 \begin{subequations}\label{exampleslemmashaperefplus:eq0}
 \begin{align}
      (\bm{x},1) \in \hat C^i \quad \Leftrightarrow  \quad  \bm{x}\in C^i\\
      \set{(\bm{x},t)\in \Real^{n+1}\,:\, \norm{\bx}_2\leq \frac{r(i-1)+r(i+1)}{2}t}\subseteq \hat C^i\label{exampleslemmashaperefplus:eq1}
 \end{align}
 \end{subequations}
where \eqref{exampleslemmashaperefplus:eq1}  follows from condition \ref{exampleslemmashapepluscontainmentconditions}.

Next, let $l_i:\Real_+\to \Real_+$ be the function 
\[l_i(z)=\frac{r(i+1)-r(i-1)}{2}z-\frac{(i-1)r(i+1)-(i+1)r(i-1)}{2},\]
that describes the line through $\bra{(i-1),r(i-1)}$ and $\bra{(i+1),r(i+1)}$. This function is strictly increasing because $r$ is strictly increasing. In addition, by concavity of $r$ we have that $l_i$ satisfies 
\begin{equation}\label{exampleslemmashaperefplus:eqconcave}
    r(z)\leq l_i(z) \; \forall z\notin [i-1,i+1].
\end{equation}

Finally, let 
$\tilde C^i:=\set{(\bm{x},z)\in \Real^{n+1}\,:\,\bra{\bm{x},\frac{2l_i(z)}{r(i-1)+r(i+1)}}\in \hat C^i}$. Then, because $l_i$ is strictly increasing and  $l_i(i)=\frac{r(i-1)+r(i+1)}{2}$, \eqref{exampleslemmashaperefplus:eq0} implies 
 \begin{subequations}\label{exampleslemmashaperefplus:eq10}
 \begin{align}
      (\bm{x},i) \in \tilde C^i \quad \Leftrightarrow  \quad  \bm{x}\in C^i\label{exampleslemmashaperefplus:eq1m1}\\
      \set{(\bm{x},z)\in \Real^{n+1}\,:\, \norm{\bx}_2\leq l_i(z)}\subseteq \tilde C^i\label{exampleslemmashaperefplus:eq11}
 \end{align}
 \end{subequations}
We claim that for all $i,j\geq 1$ with $i\neq j$ we have
\begin{equation}\label{exampleslemmashaperefplus:validity}
  C^j\times\set{j} = \set{(\bm{x},z)\in \tilde C^j\,:\, z=j}\subseteq \tilde C^i.
\end{equation}
The first equality follows directly from \eqref{exampleslemmashaperefplus:eq1m1}. For the second containment,  first note that \eqref{exampleslemmashaperefplus:eqconcave} and \eqref{exampleslemmashaperefplus:eq11} imply
  \[\set{(\bm{x},z)\in \Real^{n+1}\,:\, \begin{alignedat}{3}\norm{\bx}_2&\leq r(z),\\ z&\leq i-1\end{alignedat}}\cup \set{(\bm{x},z)\in \Real^{n+1}\,:\, \begin{alignedat}{3}\norm{\bx}_2&\leq r(z),\\ z&\geq  i+1\end{alignedat}} \subseteq  \tilde C^i,\]
The claim then follows by noting that the second containment in condition \ref{exampleslemmashapepluscontainmentconditions} implies that for any $j\notin (i-1,i+1)$ we have 
\[
C^j\times\set{j} \subseteq \set{(\bm{x},z)\in \Real^{n+1}\,:\, \begin{alignedat}{3}\norm{\bx}_2&\leq r(z),\\ z&\leq i-1\end{alignedat}}\cup \set{(\bm{x},z)\in \Real^{n+1}\,:\, \begin{alignedat}{3}\norm{\bx}_2&\leq r(z),\\ z&\geq  i+1\end{alignedat}}.
\]

  Then, \eqref{exampleslemmashaperefplus:eq1m1} and \eqref{exampleslemmashaperefplus:validity} imply that  a rational MICP formulation of $S=\bigcup_{i=1}^\infty \bra{C^i}$ is given by
  \[z\in \mathbb{Z},\quad z\geq 1,\quad (\bm{x},z) \in \tilde C^i\quad \forall i\geq 1.\]
In turn this implies that a rational MICP formulation of $S=\bigcup_{i=1}^\infty \bra{C^i+i\bm{e}(1)}$ is given by
  \[z\in \mathbb{Z},\quad z\geq 1,\quad (\bm{x}-z\bm{e}(1),z) \in \tilde C^i\quad \forall i\geq 1.\]
  
The mutually disjoint property is direct from $r(t)< 1/2 $ and the second containment in condition \ref{exampleslemmashapepluscontainmentconditions}. 
\qed
\end{proof}

\autoref{exampleslemmashaperefplus}  essentially allows us to pick arbitrary shapes for the convex sets $A_{\bm z}$, as, up to translations, they are equal to the predefined $C^i$, as long as they include and are included in balls of appropriate radii. In particular, the following corollary shows how we can employ \autoref{exampleslemmashaperefplus} to achieve this for arguably one of the simplest families of infinitely many shapes; we prove that the convex sets can be chosen to be a sequence of regular polygons with an appropriately increasing number of sides. In particular, any two such polygons are not combinatorially equivalent.
\begin{restatable}{corollary}{exampleslemmashape}\label{exampleslemmashaperef} 
There exists an increasing function $g:\mathbb{N}\setminus\set{0}\to \mathbb{N}\setminus\set{0}$ such that  $S=\bigcup_{i=1}^\infty \bra{P^i+i\bm{e}(1)}$ is a rational MICP representable infinite union of mutually disjoint  sets and $P^i\subseteq \Real^2$ is a regular $g(i)$-sided polygon for all $i\in \mathbb{N}\setminus\set{0}$.
\end{restatable}
\begin{proof}

For each integer $i\geq 1$ let
\[r(i)=\frac{1}{2}\bra{1-\frac{1}{i+1}},\quad g(i)=\left\lceil\pi\bigg/\bra{\arccos \left(\frac{(i+1) \left(i^2+i-1\right)}{i^2 (i+2)}\right)}\right \rceil\]
and $P^i=\conv\bra{\set{r(i)\cos(\bra{2 k \pi/g(i) }, r(i)\sin\bra{2 k \pi/g(i)}}_{k=1}^{g(i)}}$ be a regular $g(i)$-sided polygon inscribed in the circle of radius $r(i)$ and center the origin.

Now recall the folklore euclidean geometry result that if $\varepsilon = \frac{1}{\cos\bra{\pi/m}}-1$, then a regular $m$-sided polygon inscribed in the circle of radius $1+\varepsilon$ contains the circle of radius $1$ (e.g. \cite[Theorem 2.1]{glineur2000}).
Using the result, if we let $\varepsilon(i)=\frac{2r(i)-r(i-1)-r(i+1)}{r(i-1)+r(i+1)}$ so that $r(i)=(1+\varepsilon(i))\bra{\frac{r(i-1)+r(i+1)}{2}}$, we have that for any
\[m\geq \left\lceil\pi\bigg/\bra{\arccos \left(\frac{1}{1+\varepsilon(i)}\right)}\right \rceil\]
the regular $m$-sided polygon inscribed in the circle of radius $r(i)$ contains the circle of radius $\bra{\frac{r(i-1)+r(i+1)}{2}}$. The result then follows from \autoref{exampleslemmashaperefplus} for $C^i=P^i$ and our appropriate definition of the functions $r(i), g(i).$

\qed
\end{proof}
\subsection{Equal volume implies finitely many shapes}

One interesting aspect of \autoref{exampleslemmashaperef} is that the volume of the $\bm z$-projected sets of the example varies. In this subsection we prove that this property is of \textit{fundamental importance} if the $\bz$-projected sets of a general-integer MICP-R set are to have infinitely many shapes. Our connection between volumes of  $\bz$-projected sets and their shapes is based on the following technical result, which is a consequence of the Brunn-Minkowski inequality \cite{schneider2014convex, KlainBMequality}, arguably one of the centerpieces of modern convex geometry.

\begin{lemma}\label{volumelemma}
Let  $M\subseteq \Real^{n+p+d}$ be a closed convex set inducing an MICP formulation of $S\subseteq \Real^n$, $I$ be its index set and $\set{A_{\bz}}_{\bz\in I}$ be its $\bz$-projected sets. Then $h:I \rightarrow \mathbb{R}$ defined by $h(\bm{z})=(\mathrm{Vol}(A_{\bm{z}}))^{\frac{1}{n}}$ is a concave function. Furthermore, for any $\bm{z},\bm{w} \in I$ and $\lambda \in [0,1]$ it holds:
$$h(\lambda \bm{z}+(1-\lambda) \bm{w}) \geq \mathrm{Vol}(\lambda A_{\bm{z}}+(1-\lambda)A_{\bm{w}})^{\frac{1}{n}} \geq \lambda h(\bm{z})+(1-\lambda)h(\bm{w}),$$
\end{lemma}

Before presenting the proof of Lemma \ref{volumelemma} we note that the additive operation between sets in the statements corresponds to Minkowski addition.
\begin{proof}
First note that convexity of $M$ implies that for any $\bm{z},\bm{w} \in I$ and $\lambda \in [0,1]$ by we have that the set 
\[
\lambda \bra{M \cap (\mathbb{R}^{n+p} \times \{\bm{z}\})}+(1-\lambda)\bra{M \cap (\mathbb{R}^{n+p} \times \{\bm{w}\})}\]
is contained in $\bra{M \cap (\mathbb{R}^{n+p} \times \{\lambda \bm{z}+(1-\lambda)\bm{w}\})}$.
Then by recalling that $A_{\bm{z}}=\proj_{\bx}\bra{M \cap (\mathbb{R}^{n+p} \times \{\bm{z}\})}$ and noting that $\proj_{\bx}$ preserves containment of sets, we further have $\lambda A_{\bm{z}}+(1-\lambda)A_{\bm{w}} \subseteq A_{\lambda {\bm{z}}+ (1-\lambda) \bm{w}}$. Hence we have $$h(\lambda \bm{z}+(1-\lambda)\bm{w})=(\mathrm{Vol}(A_{\lambda \bm{z}+(1-\lambda)\bm{w}}))^{\frac{1}{n}} \geq \mathrm{Vol}(\lambda A_{\bm{z}}+(1-\lambda)A_{\bm{w}})^{\frac{1}{n}}.$$ But now by the Brunn-Minkowski inequality \cite[Theorem 6.1.1]{schneider2014convex}  and the elementary equality $\mathrm{Vol(cA)}=c^n \mathrm{Vol}(A)$ for any Borel set $A$ and $c>0,$ we have $$\mathrm{Vol}(\lambda A_{\bm{z}}+(1-\lambda)A_{\bm{w}})^{\frac{1}{n}} \geq \lambda \mathrm{Vol}( A_{\bm{z}})^{\frac{1}{n}}+(1-\lambda)\mathrm{Vol}(A_{\bm{w}})^{\frac{1}{n}}.$$ The above two inequalities and the definition of $h$ imply $$h(\lambda \bm{z}+(1-\lambda)\bm{w}) \geq \mathrm{Vol}(\lambda A_{\bm{z}}+(1-\lambda)A_{\bm{w}})^{\frac{1}{n}} \geq  \lambda h(\bm{z})+(1-\lambda)h(\bm{w}),$$ as needed. \qed
\end{proof}

We conclude with our final result that, unless the volumes of the $\bm z$-projected sets vary, the sets correspond to finitely many shapes under the shape-equivalency notion of translation invariance.

\begin{restatable}{theorem}{voltheo}\label{voltheolabel}
If $S$ has an MICP formulation such that all its $\bm{z}$-projected sets have the same volume, then there exists a finite family of convex sets $\set{T_i}_{i=1}^m,$ for some $m \in \mathbb{N},$  such that all  $\bm{z}$-projected sets are translations of sets in this family.
\end{restatable}
\begin{proof}
Let  $M\subseteq \Real^{n+p+d}$ be a closed convex set inducing an  MICP formulation of $S\subseteq \Real^n$, such that its $\bz$-projected sets  $\set{A_{\bz}}_{\bz\in I}$ all have the same volume. Let $I$ be the index set of this formulation.
For $h$ defined in Lemma~\ref{volumelemma} we have that there exists $\alpha>0$ such that $h(\bm{z})=\alpha$ for all $\bm{z}\in I\cap \mathbb{Z}^d$.
We claim for  any two $\bm{z},\bm{w} \in I\cap \mathbb{Z}^d$ with with the same modulo 2 pattern in their coordinates (i.e. $(\bm{z}+\bm{w})/2\in \mathbb{Z}^d$) $A_{\bm{z}}$ is a translation of $A_{\bm{w}}$.
Indeed, we have $h(\bm{z})=h(\bm{w})=h((\bm{z}+\bm{w})/2)=\alpha$, which implies
$$\frac{1}{2}h(\bm{z})+\frac{1}{2}h(\bm{w})=h\bra{\frac{\bm{z}+\bm{w}}{2}}=\alpha.$$
Together with Lemma~\ref{volumelemma}  this implies $$\mathrm{Vol}\bra{\frac{1}{2} A_{\bm{z}}+\frac{1}{2}A_{\bm{w}}}^{\frac{1}{n}}=\frac{1}{2}\mathrm{Vol}( A_{\bm{z}})^{\frac{1}{n}}+\frac{1}{2}\mathrm{Vol}(A_{\bm{w}})^{\frac{1}{n}}=\alpha.$$ But this implies equality in the Brunn-Minkowski inequality for the convex sets  $A_{\bm{z}}, A_{\bm{w}}$ which implies that they are homothetic (see e.g. \cite{KlainBMequality}). Since the sets are assumed to have the same volume, this implies our translation claim.

Finally, the result follows by defining the family of sets $\set{T_i}_{i=1}^m$ for some  $m \leq 2^d,$ which includes one representative $A_{\bm{z}}$ for each of the finite number of modulo 2 patterns that appear for some $\bm{z}\in I\cap \mathbb{Z}^d$.\qed
\end{proof}

\bibliography{main}

\end{document}